\documentclass[a4paper]{amsart}
\usepackage{amssymb}
\usepackage{ifthen}
\usepackage{graphicx}

\newtheorem{thm}{Theorem}

\theoremstyle{definition}

\newcommand{\A}{{\mathcal A}}

\newcommand{\es}{{\mathcal S}}

\newcommand{\D}{{\mathbb D}}
\newcommand{\real}{{\operatorname{Re}\,}}

\begin{document}
\bibliographystyle{amsplain}

\title[On the difference  the two initial logarithmic coefficients ]{On the difference of the two initial logarithmic coefficients for Bazilevi\v{c} class of univalent functions}

\author[M. Obradovi\'{c}]{Milutin Obradovi\'{c}}
\address{Department of Mathematics,
Faculty of Civil Engineering, University of Belgrade,
Bulevar Kralja Aleksandra 73, 11000, Belgrade, Serbia.}
\email{obrad@grf.bg.ac.rs}

\author[N. Tuneski]{Nikola Tuneski}
\address{Department of Mathematics and Informatics, Faculty of Mechanical Engineering, Ss. Cyril and
Methodius
University in Skopje, Karpo\v{s} II b.b., 1000 Skopje, Republic of North Macedonia.}
\email{nikola.tuneski@mf.edu.mk}

\subjclass[2020]{30C45, 30C50}
\keywords{univalent functions, logarithmic coefficient, sharp results,Bazilevi\v{c} class }

\begin{abstract}
In this paper we give sharp bounds of the difference of the moduli of the second and the first logarithmic coefficient for Bazilevi\v{c} class of univalent functions.
\end{abstract}

\maketitle

\section{Introduction and definitions}

\medskip

Let $\A$ denote the class of analytic functions $f$ in the open unit disk $\D=\{z:|z|<1\}$ normalized with $f(0)=f'(0)-1=0$, i.e.,let
\begin{equation}\label{e1}
  f(z)=z+a_2z^2+a_3z^3+\cdots.
\end{equation}
 The functions from $\A$ that are one-on-one and onto are the well known univalent functions, and the corresponding class is denoted by  $\mathcal{S}$.

 For $f\in \mathcal{S}$ let
\begin{equation}\label{eq2}
\log\frac{f(z)}{z}=2\sum_{n=1}^\infty \gamma_n z^n,
\end{equation}
where $\gamma_n$, $n=1,2,\ldots$ are the logarithmic coefficients of the functions $f$.
Those coefficients are in the connection with the famous Bieberbach conjecture from 1916 (\cite{bieber}), stating that for all functions from $\mathcal{S}$, $|a_n|\leq n$, $n=2,3,\ldots$,  proven by de Branges in 1985 (\cite{branges}).

For the Koebe function we have $\gamma_{n}=1/n$, but the natural conjecture $|\gamma_n|\leq1/n$ ($n=2,3,\ldots$), is not true in general (see \cite[Section 8.1]{duren}, \cite{book}).

We note that from the relations \eqref{e1} and \eqref{eq2}, after  comparing of coefficients, we receive
\begin{equation}\label{eq3}
 \gamma_{1}=\frac{a_{2}}{2}\quad \text{and}\quad \gamma_{2}=\frac{1}{2}\left(a_3-\frac{1}{2}a_{2}^{2}\right).
\end{equation}
For the general class $\mathcal{S}$ the sharp estimates of single logarithmic coefficients are known only for $\gamma_1$ and $\gamma_2$, namely,
$$|\gamma_1|\leq1\quad\mbox{and}\quad |\gamma_2|\le \frac12+\frac1e=0.635\ldots.$$
In their papers \cite{OT_2023-6} and \cite{OT_2025-1}, the authors gave the next estimates for the class $\mathcal{S}$:
$$|\gamma_{3}|\leq 0.556617\ldots , \quad |\gamma_{4}|\leq0.51059\ldots  .$$

Another approach of study of logarithmic coefficients is by considering the difference between the modulus of two consecutive coefficients over the general class $\mathcal{S}$, or over its subclasses.
In that direction, the sharp estimates of $|\gamma _{2}|-|\gamma_{1}|$ for the general class $\mathcal{S}$ were given in \cite{lecko}, with a much simpler proof in \cite{OT_2023-3} obtained using different method.
Recently, the authors gave the appropriate results for some subclasses of univalent functions (see \cite{OT_2025}).
\bigskip

\noindent
\textbf{Theorem A.}   \textit{For every $f\in\es$, $-\frac{\sqrt2}{2}\leq |\gamma_{2}|-|\gamma_{1}|\leq\frac12$ holds sharply.}

The differences $|\gamma _{3}|-|\gamma_{2}|$ and $|\gamma _{4}|-|\gamma_{3}|$, again for the general class $\mathcal{S}$ were studied in \cite{OT_2023-5}.

In this paper we give sharp estimates of the difference $|\gamma_2|-|\gamma_1|$, for the class $\mathcal{B}_{1}(\alpha), \alpha>0$, defined by
\[
Re\left[\left(\frac{f(z)}{z}\right)^{\alpha-1}f'(z)\right]>0,\quad \alpha>0, \,z\in \D.
\]
The class $\mathcal{B}_{1}(\alpha), \alpha>0$, is a subclass of the class of univalent functions $\mathcal{S}$ and type of Bazilevi\v{c} class (see \cite{baz} and \cite{Singh}).

\medskip

\section{Main results}

In this section we will prove sharp bounds for the difference $|\gamma_2|-|\gamma_1|$ using similar methods as in \cite{OT_2023-3}.

\begin{thm}
Let $f\in\mathcal{B}_{1}(\alpha)$ for some $\alpha>0$. Then
$$ -\frac{1}{\sqrt{(\alpha+1)^{2}+1}}\leq |\gamma_{2}|-|\gamma_{1}|\leq\frac{1}{\alpha+2}.$$
The right estimation is sharp for all $\alpha>0$, while the left is sharp for $\alpha_1\le\alpha\le\alpha_2$, where $\alpha_1= \frac{1}{2} \left(\sqrt{6}-2\right)=0.2247\ldots$ and $\alpha_2 = \sqrt2-1=0.4142\ldots$ are the positive roots of the equation
$$2 \alpha ^4+8 \alpha ^3+5 \alpha ^2-6 \alpha +1=0.$$
\end{thm}

\begin{proof}
Let $f\in\mathcal{B}_{1}(\alpha),$ $ \alpha>0$, and let
$$F(f,\alpha;z)=\left(\frac{f(z)}{z}\right)^{\alpha-1}f'(z).$$
Then, by the definition of the class $\mathcal{B}_{1}(\alpha)$, the function
\[
\begin{split}
G(f,\alpha;z) &=\frac{1-F(f,\alpha;z)}{1+F(f,\alpha;z)}\\
&=-\frac{\alpha+1}{2}a_{2}z-
\left[\frac{\alpha+2}{2}a_{3}-\frac{\alpha+3}{4}a_{2}^{2}\right]z^{2}+\cdots \\
&=:c_{1}z+c_{2}z^{2}+\cdots
\end{split}
\]
is a Schwartz function, and by using the inequalities $|c_{1}|\leq1$ and $|c_{2}|\leq 1-|c_{1}|^{2},$
 we receive
 $$\frac{1}{2}(1+\alpha)|a_{2}|\leq1$$
  and
$$\left|\frac{\alpha+2}{2}a_{3}-\frac{\alpha+3}{4}a_{2}^{2}\right|\leq1-\frac{1}{4}(1+\alpha)^{2}|a_{2}|^{2}.$$
From these inequalities we obtain  $|a_{2}|\leq\frac{2}{1+\alpha}$ and
\begin{equation}\label{eq5}
\left|a_{3}-\frac{\alpha+3}{2(\alpha+2)}a_{2}^{2}\right|
\leq\frac{2}{\alpha+2}-\frac{(\alpha+1)^{2}}{2(\alpha+2)}|a_{2}|^{2}.
\end{equation}

\medskip

For the upper bounds (right estimates) of $|\gamma _{2}|-|\gamma_{1}|$, using \eqref{eq3} and \eqref{eq5}, for
$\alpha>0$ we have:
\[
\begin{split}
|\gamma _{2}|-|\gamma_{1}|&=\frac{1}{2}\left|a_3-\frac{1}{2}a_{2}^{2}\right|-\frac{1}{2}|a_{2}|\\
&=\frac{1}{2}\left|\left( a_{3}-\frac{1}{2}a_{2}^{2}-\frac{1}{2(\alpha+2)}a_{2}^{2} \right) +
\frac{1}{2(\alpha+2)}a_{2}^{2}\right|-\frac{1}{2}|a_{2}|\\
&\leq\frac12\left|a_{3}-\frac{\alpha+3}{2(\alpha+2)}a_{2}^{2}\right|+\frac{1}{4(\alpha+2)}|a_{2}|^{2}
-\frac{1}{2}|a_{2}|\\
&\leq \frac{1}{\alpha+2}-\frac{(\alpha+1)^{2}}{4(\alpha+2)}|a_{2}|^{2}+\frac{1}{4(\alpha+2)}|a_{2}|^{2}
-\frac{1}{2}|a_{2}|\\
&=\frac{1}{\alpha+2}-\frac{\alpha}{4}|a_{2}|^{2}-\frac{1}{2}|a_{2}|\\
&\leq \frac{1}{\alpha+2}.
\end{split}
\]
As for sharpness of this estimate, let first consider relation
\begin{equation}\label{eq6}
\left( \frac{f(z)}{z}\right)^{\alpha-1}f'(z)=h(z),\quad \alpha>0, h(0)=1.
\end{equation}
From \eqref{eq6} we receive
\[ [f^\alpha(z)]' = \alpha z^{\alpha-1}h(z),\]
and after some calculations,
\begin{equation}\label{eq7}
 f(z)=z \left( \frac{\alpha}{z^{\alpha}}\int_{0}^{z}t^{\alpha-1}h(t)dt\right)^{\frac{1}{\alpha}}.
\end{equation}
Choosing $h(z)=\frac{1+z^{2}}{1-z^{2}}=1+2z^2+\cdots$, after equating the coefficients, we receive $a_2=0$ and $a_3=\frac{2}{\alpha+2}$, i.e.,
$$f(z)=z+\frac{2}{\alpha+2}z^{3}+\cdots .$$
This gives that
$$|\gamma _{2}|-|\gamma_{1}|=\frac{1}{2}|a_{3}|=\frac{1}{\alpha+2}.$$
Thus, the right hand estimation is sharp for every $\alpha>0$.

\medskip

 Now, for the lower bound we should prove
$$ \frac{1}{2}\left|a_3-\frac{1}{2}a_{2}^{2}\right|-\frac{1}{2}|a_{2}|\geq-\frac{1}{\sqrt{(\alpha +1)^{2}+1}},$$
i.e.,
\begin{equation}\label{eq8}
\left|a_3-\frac{1}{2}a_{2}^{2}\right|\geq|a_{2}|-\frac{2}{\sqrt{(\alpha +1)^{2}+1}}.
\end{equation}
\medskip

If $0\leq|a_{2}|<\frac{2}{\sqrt{(\alpha +1)^{2}+1)}}$, then the previous inequality obviously holds.

\medskip

Further, let study the case when $|a_{2}|\geq \frac{2}{\sqrt{(\alpha +1)^{2}+1)}}.$ Using
\eqref{eq3} and \eqref{eq5} we have
\[
\begin{split}
\left|a_3-\frac{1}{2}a_{2}^{2}\right|&=
\left|\left(a_{3}-\frac{\alpha+3}{2(\alpha+2)}a_{2}^{2}\right)+
\frac{1}{2(\alpha+2)}a_{2}^{2}\right|\\
&\geq\frac{1}{2(\alpha+2)}|a_{2}|^{2}-\left|a_{3}-\frac{\alpha+3}{2(\alpha+2)}a_{2}^{2}\right|\\
&\geq\frac{1}{2(\alpha+2)}|a_{2}|^{2}-\frac{2}{\alpha+2}+\frac{(\alpha+1)^{2}}{2(\alpha+2)}|a_{2}|^{2}\\
&=\frac{(\alpha+1)^{2}+1}{2(\alpha+2)}|a_{2}|^{2}-\frac{2}{\alpha+2}.
\end{split}
\]
So, it is enough to prove that
$$\frac{(\alpha+1)^{2}+1}{2(\alpha+2)}|a_{2}|^{2}-\frac{2}{\alpha+2}\geq
|a_{2}|- \frac{2}{\sqrt{(\alpha +1)^{2}+1}}, $$
which is equivalent to
$$\frac{(\alpha+1)^{2}+1}{2(\alpha+2)}\left(|a_{2}|^{2}- \frac{4}{(\alpha +1)^{2}+1}\right)
-\left(|a_{2}|- \frac{2}{\sqrt{(\alpha +1)^{2}+1}}\right)\geq0 ,$$
i.e., to
$$\left(|a_{2}|- \frac{2}{\sqrt{(\alpha +1)^{2}+1}}\right)\left[\frac{(\alpha+1)^{2}+1}{2(\alpha+2)}   \left(|a_{2}|+\frac{2}{\sqrt{(\alpha +1)^{2}+1}}\right)-1\right]\geq 0.$$
The last inequality indeed holds since by assumption $|a_{2}|\geq \frac{2}{\sqrt{(\alpha +1)^{2}+1}}$
and
\[\begin{split}
\frac{(\alpha+1)^{2}+1}{2(\alpha+2)}\left(|a_{2}|+\frac{2}{\sqrt{(\alpha +1)^{2}+1}}\right)-1
&\geq \frac{(\alpha+1)^{2}+1}{2(\alpha+2)}\frac{4}{\sqrt{(\alpha +1)^{2}+1}}-1 \\
&=\frac{2\sqrt{(\alpha +1)^{2}+1}}{\alpha+2}-1\\
&>\frac{2(\alpha+1)}{\alpha+2}-1=\frac{\alpha}{\alpha+2}>0.
\end{split}\]

\medskip

As for the sharpness of the lower (left) estimate, equality sign in \eqref{eq8} is obtained if $|a_{2}|=\frac{2}{\sqrt{(\alpha +1)^{2}+1}}=:b_{2}$, i.e., if
$ a_{2}=b_{2}e^{i\theta}$, and $a_3=\frac{1}{2}a_{2}^{2}=\frac{1}{2}b_{2}e^{i2\theta}$. Now, the function
$$f(z)=z+b_{2}e^{i\theta}z^{2}+\frac{1}{2}b_{2}^{2}e^{2i\theta}z^{3}$$
has such coefficients $a_2$ and $a_3$, and further
\[\left( \frac{f(z)}{z}\right)^{\alpha-1}f'(z) = 1+(\alpha +1)b_{2}e^{i\theta}z+\frac{\alpha(\alpha+2)}{2}b^2_{2}e^{i2\theta}z^{2} =: h_2(z),\]
It remains to check that under the conditions of the theorem, i.e., for $\alpha_1\le \alpha\le\alpha_2$, $\real\{h_2(z)\}>0,\, z\in\D$.
In that direction, we study
$$\real\{h_2(e^{i\phi})\}=1+(\alpha +1)b_{2}\cos(\theta+\phi)+\frac{\alpha(\alpha+2)}{2}b_{2}^{2}\cos(2(\theta+\phi)),$$
$-\pi\le \phi \le\pi$, which with notation $\cos(\theta+\phi)=t$, $-1\leq t\leq1$, becomes
$$\real\{h_2(e^{i\phi})\} = \alpha(\alpha+2)b_{2}^{2}t^{2}+(\alpha+1)b_{2}t+1-\frac{\alpha(\alpha+2)}{2}=: \psi(t),$$
and we need to show that $\psi(t)\ge0$ for all $\alpha_0\le\alpha\le\alpha_2$ and $-1\le t\le 1$. The vertex of the concave upwards quadratic function $\psi$ lies on the left half-plane, so on $[-1,1]$ it attains its minimal value at $t_* = \max\{-1, t_1\}$,
\[ t_1 =  -\frac{\left((\alpha +1)^2+1\right) (\alpha +1)}{4 \alpha  (\alpha +2) \sqrt{(\alpha +1)^2+1}} . \]
Now, $t_1\ge -1$ is equivalent to $t_1^2 \le 1$, i.e., to
\[ \frac{(\alpha +1)^2 \left((\alpha +1)^2+1\right)}{16 \alpha ^2 (\alpha +2)^2} \le 1,\]
and further
\[ 2 + 6 \alpha  - 57 \alpha ^2 - 60 \alpha ^3 - 15 \alpha ^4 \le 0.\]
The last, for $\alpha>0$, is equivalent to $\alpha\ge \alpha_3 =  \sqrt{\frac{1}{30} \left(\sqrt{129}+33\right)}-1 = 0.21597\ldots$.

Thus, for $\alpha\ge \alpha_3$, the minimal value of $\psi$ on the interval $[-1,1]$ is
\[\psi(t_*) = -\frac{(\alpha +1)^2}{4 \alpha  (\alpha +2)}-\frac{1}{2} \alpha  (\alpha +2)+1\]
which is greater or equal to zero if, and only, if $\alpha_3<\alpha_1\le\alpha\le\alpha_2$.

\medskip

The case when $\alpha\le \alpha_3$ leads to the minimal value of $\psi$ on the interval $[-1,1]:$
\[\psi(-1) = -\frac{2 (\alpha +1)}{\sqrt{(\alpha +1)^2+1}}-\frac{1}{2} \alpha  (\alpha +2)+\frac{4 \alpha  (\alpha +2)}{(\alpha +1)^2+1}+1,\]
which is negative and does not bring sharpness.
\end{proof}

\end{document}